\documentclass[11pt]{article}

\usepackage{amsmath}
\usepackage{latexsym}
\usepackage{amsfonts}
\usepackage{amssymb}
\usepackage{amscd}
\usepackage{theorem}
\usepackage{a4wide}

\newtheorem{theorem}{Theorem}[section]
\newtheorem{lemma}[theorem]{Lemma}
\newtheorem{proposition}[theorem]{Proposition}
\newtheorem{corollary}[theorem]{Corollary}

{\theorembodyfont{\rmfamily}
\newtheorem{definition}[theorem]{Definition}

  \newtheorem{remark}[theorem]{Remark}
}

\newenvironment{proof}{    
  \noindent
  \textbf{Proof.}}{
  \hfill $\Box$
  \vspace{3mm}
}

\numberwithin{equation}{section}

\newcommand{\N}{\mathbb{N}} 
\newcommand{\Z}{\mathbb{Z}} 
\newcommand{\R}{\mathbb{R}} 
\newcommand{\C}{\mathbb{C}} 
\newcommand{\K}{\mathbb{K}} 
\newcommand{\D}{\mathbb{D}} 

\DeclareMathOperator{\IL}{ind}

\begin{document}

\title{\textbf{Shrinking and boundedly complete atomic decompositions in Fr\'echet spaces}}
\author{\textbf{Jos\'e Bonet, Carmen Fern\'andez, Antonio Galbis, Juan M. Ribera}}
\date{}

\maketitle

\begin{abstract}
We study atomic decompositions in Fr\'echet spaces and their duals, as well as perturbation results. We define shrinking and boundedly complete atomic decompositions on a locally convex space, study the duality of these two concepts and their relation with the reflexivity of the space. We characterize when an
unconditional atomic decomposition is shrinking or boundedly complete in terms of properties of the space. Several examples of concrete atomic decompositions in function spaces  are also presented.
\end{abstract}

\renewcommand{\thefootnote}{}


\section*{Introduction}
Atomic decompositions are used to represent an arbitrary element $x$ of a function space $E$ as a series expansion involving a fixed countable set $(x_j)_j$ of elements in that space such that the coefficients of the expansion of $x$ depend in a linear and continuous way on $x$. Unlike Schauder bases, the expression of an element $x$ in terms of the sequence $(x_j)_j$, i.e.\ the reproduction formula for $x$, is not necessarily unique. Atomic decompositions appeared in applications to signal processing and sampling theory among other areas. Feichtinger characterized Gabor atomic decomposition for modulation spaces \cite{016} and the general theory was developed in his joint work with Gr\"{o}chenig \cite{013} and \cite{014}. In these papers, the authors show that reconstruction through atomic decompositions is not limited to Hilbert spaces. Indeed, they obtain atomic decompositions for a large class of Banach spaces, namely the coorbit spaces. Atomic decompositions are a less restrictive structure than bases,
because a complemented subspace of a Banach space with basis has always a natural atomic decomposition, that is  obtained from the
basis of the superspace. Even when the complemented subspace has a basis, there is not a systematic way to find it. There is a vast literature  dedicated to the subject. The related topic of frame expansions in Banach spaces was considered for example in  \cite{003} and \cite{017}.

Carando, Lasalle and Schmidberg \cite{005} and \cite{004} studied atomic decompositions and their relationship with duality and reflexivity of Banach spaces. They extended the concepts of shrinking and boundedly complete Schauder basis to the atomic decomposition framework. They considered when an atomic decomposition for a Banach space generates, by duality, an atomic decomposition for its dual space and characterized the reflexivity of a Banach space in terms of properties of its atomic decompositions. Unconditional atomic decompositions allowed them to prove James-type results characterizing shrinking and boundedly complete unconditional atomic decompositions in terms of the containment in the Banach space of copies of $\ell_1$ and $c_0$ respectively.

Very recently, Pilipovic and Stoeva \cite{028} (see also \cite{034}) studied series expansions in (countable) projective or inductive limits of Banach spaces. In this article we begin a systematic study of atomic decompositions in locally convex spaces, but our main interest lies in Fr\'echet spaces and their duals. The main difference with respect to the concept considered in \cite{028} is that our approach does not depend on a fixed representation of the Fr\'echet space as a projective limit of Banach spaces. We mention the following preliminary example as a motivation for our work: Leontiev proved that for each bounded convex domain $G$ of the complex plane ${\mathbb C}$ there is a sequence of complex numbers $(\lambda_j)_j$ such that every holomorphic function $f \in H(G)$ can be expanded as a series of the form $f(z)=\sum_{j=1}^\infty a_j e^{\lambda_j z}$, converging absolutely and uniformly on the compact subsets of $G$. It is well-known that this expansion is not unique. We refer the reader e.g. to
Korobeinik's survey \cite{043}. A priori it is not clear whether the coefficients $a_j$ in the expansion can be selected depending continuously on the function $f$. However,
Korobeinik and Melikhov \cite[Th. 4.3 and remark 4.4(b)]{032} showed that this is the case when the boundary of the open set $G$ is of class $C^2$; thus obtaining what we call below an unconditional atomic decomposition for the Fr\'echet space $H(G)$. These are the type of phenomena and reproducing formulas that we try to understand in our paper.

Our main purpose is to investigate the relation between the properties of an existing atomic decomposition in a Fr\'echet space $E$ and the structure of the space, for example if $E$  is reflexive or if it contains copies of $c_0$ or $\ell_1$. For complete barrelled spaces, we show in \ref{thrm01} that having an atomic decomposition is equivalent to being complemented in a complete locally convex space with a Schauder basis. Perturbation results for atomic decompositions are given in Theorem \ref{perturb}. We introduce shrinking and boundedly complete atomic decompositions on a locally convex space, study the duality of these two concepts and their relation with the reflexivity of the space; see Theorem \ref{prop07}. Unconditional atomic decompositions are studied in Section \ref{uncond}. We completely characterize, for a given unconditional atomic decomposition, when it is shrinking or boundedly complete in terms of properties of the space in Theorems \ref{thrm21} and \ref{boundedlyLF}. As a tool, that
could be of independent interest, we show Rosenthal $\ell_1$ Theorem for boundedly retractive inductive limits of Fr\'echet spaces; see Proposition \ref{prop30}. Some examples of concrete atomic decompositions in function spaces  are also included in Section \ref{Examples}. Our Theorem \ref{th01} shows a remarkable  relation between the existence of a continuous linear extension operator for $C^\infty$ functions defined on a compact subset $K$ of $\R^n$ and the existence of an unconditional atomic decomposition in $C^\infty(K)$ using exponentials.

\section{Atomic decompositions in locally convex spaces}
Throughout this work, $E$  denotes a locally convex Hausdorff linear topological space (briefly, a lcs) with additional hypotheses added as needed and $cs(E)$ is the system of continuous seminorms describing the topology of $E.$ The symbol $E'$ stands for the topological dual of $E$ and $\sigma(E',E)$ for the weak* topology on $E'$. We set  $E'_\beta$ for the dual $E'$ endowed with the topology $\beta(E',E)$ of uniform convergence on the bounded sets of $E.$ We will refer to $E'_\beta$ as the strong dual of $E.$ The bidual $E''$ of $E$ is the dual of $E'_\beta$. Basic references for lcs are \cite{006} and \cite{044}. If $T:E \rightarrow F$ is a continuous linear operator, its transpose is denoted by $T':F' \rightarrow E'$, and it is defined by $T'(v)(x):=v(T(x)), x \in E, v \in F'$.  We recall that a Fr\'echet space is a complete metrizable lcs. An $\left(LF\right)$-space is a lcs that can be represented as an  inductive limit of a sequence $\left(E_n\right)_n $ of Fr\'echet spaces, and in case all the
spaces $E_n$ are Banach spaces, we call it an $\left(LB\right)$-space. In most of the results we need the assumption that the lcs is  barrelled. The reason is that Banach-Steinhaus theorem holds for barrelled lcs. Every Fr\'echet space and every $\left(LB\right)$-space is barrelled.   We refer the reader to  \cite{006} and \cite{010} for more information about barrelled spaces.

\begin{definition} Let $E$ be a lcs, $\{x_j\}_{j = 1}^{\infty}\subset E$ and $\{x_j'\}_{j = 1}^{\infty}\subset E' $. We say that $\left( \{ x_j' \} , \{ x_j \} \right)$ is an \textit{atomic decomposition of $E$} if
\begin{equation*}
 x = \sum_{j = 1}^{\infty} x_j' \left( x \right) x_j , \quad \mbox{ for all } x \in E,
\end{equation*}
the series converging in $E$.
\end{definition}

\par\medskip\noindent
A lcs $E$ which admits an atomic decomposition is separable. Let $E$ be a lcs with a Schauder basis $\{e_j\}_{j = 1}^{\infty} \subset E$ and let  $\{e_j'\}_{j = 1}^{\infty} \subset E'$ denote the coefficient functionals. Clearly,  $ \left( \{ e_j' \} , \{ e_j \} \right) $ is an atomic decomposition for $E$. The main difference with Schauder basis is that, in general, one may have a sequence  $\{x_j\}_{j = 1}^{\infty} \subset E$ and two different sequences $\{x_j'\}_{j = 1}^{\infty}\subset E'$ and $\{y_j'\}_{j = 1}^{\infty}\subset E'$ so that both $\left( \{ x_j' \} , \{ x_j \} \right)$ and $\left( \{ y_j' \} , \{ x_j \} \right)$ are atomic decompositions. See the comments after Theorem \ref{th01}.

\begin{proposition} \label{prop00}
Let $E$ be a lcs and let $P: E \rightarrow E$ be a continuous linear projection. If $ \left( \{ x_j' \} , \{ x_j \} \right) $ is an atomic decomposition for $E$, then $ \left( \{ P'(x_j') \} , \{ P\left(x_j\right) \} \right) $ is an atomic decomposition for $P\left(E\right)$.

In particular, if $E$ is isomorphic to a complemented subspace of a lcs with a Schauder basis, then $E$ admits an atomic decomposition.
\end{proposition}
\begin{proof}
    Since $\langle P'(x_j'), y \rangle = \langle x_j' , P\left( y \right) \rangle = \langle x_j' , y \rangle $ for all $y \in P(E)$ and $ j \in \N $, we obtain an atomic decomposition:
      \begin{equation*}
      y = P\left(y\right) = P\left( \sum_{ j = 1}^{\infty} x_j' \left(y\right) x_j\right) = \sum_{ j = 1}^{\infty} \langle P'\left(x_j' \right), y \rangle P \left( x_j \right).
      \end{equation*}
\end{proof}

\par\medskip\noindent
As usual $\omega$ denotes the countable product $\K^{\N}$ of copies of the scalar field, endowed by the product topology, and  $\varphi$ stands for the space of sequences with finitely many non-zero coordenates. A sequence space $\bigwedge$ is a lcs which contains $\varphi$ and is continuously included in $\omega.$

\begin{lemma}\label{lem02}
    Let $\{ x_j \}$ be a fixed sequence of non-zero elements in a lcs $E$ and let us denote by $\bigwedge$ the vector space
    \begin{equation}\label{eq01}
            \bigwedge := \{ \alpha = \left(\alpha_j\right)_j \in \omega : \sum_{j = 1}^{\infty} \alpha_j x_j \mbox{ is convergent in } E \}.
        \end{equation}
    Endowed with the  system of seminorms
        \begin{equation}\label{eq01}
            {\mathcal Q} := \left\{ q_p\left(\left(\alpha_j\right)_j\right) := \sup_n p\left(\sum_{j = 1}^n \alpha_j x_j \right) , \mbox{ for all } p \in cs(E) \right\}
        \end{equation} $\bigwedge$ is a sequence space and the canonic unit vectors form a Schauder basis. If $E$ is complete, then $\bigwedge$ is complete.
    In particular, if $E$ is a Fr\'echet (resp. Banach) space, so is $\bigwedge$.
\end{lemma}
\begin{proof}
  It is routine to check that the unit vectors are a topological basis of $\bigwedge.$ Since
 $$
q \left(\sum_{i=1}^{n}\alpha_i e_i \right )\leq q \left(\sum_{i=1}^{n+m}\alpha_i e_i \right )$$ for every $q\in {\mathcal Q}$ and for all $m,n \in \N$ and $\alpha_1, \dots, \alpha_{n+m} \in \K$ we can apply \cite[14.3.6]{006} to conclude that the unit vectors are also a Schauder basis.
\end{proof}

\begin{theorem}\label{thrm01}
Let $E$ be a complete barrelled locally convex space. The following conditions are equivalent:
\begin{itemize}
\item[\rm (1)] $E$ admits an atomic decomposition.
\item[\rm (2)] $E$ is isomorphic to a complemented subspace of a complete sequence space with the canonical unit vectors as Schauder basis.
\item[\rm (3)] $E$ is isomorphic to a complemented subspace of a complete locally convex space with a Schauder basis.
\end{itemize}
In particular, a Fr\'echet space $E$ admits an atomic decomposition if and only if it is
isomorphic to a complemented subspace of a Fr\'echet space with a Schauder basis.
\end{theorem}
\begin{proof}
  $(1) \Rightarrow (2)$ Let $ \left( \{ x_j' \} , \{ x_j \} \right) $ be an atomic decomposition of $E$. We may assume that $x_j \neq 0$ for all $j \in \N$. Let $\bigwedge$ be the complete lcs of sequences defined as in Lemma \ref{lem02}. We define $F_n:E\longrightarrow E$ as $F_n\left(x\right):= \sum_{ j = 1 }^n x_j'\left(x\right) x_j$. Since $E$ is barrelled the sequence $\left(F_n \right)_n$ is equicontinuous, that is, for every $p \in cs(E)$ there exists $p' \in cs(E)$ such that $p\left(F_n \left( x \right) \right) \leq p'\left(x\right)$ for every $ x \in E$ and for every $n \in \N.$ Consequently the map $U:E \longrightarrow \bigwedge,\ U\left(x\right) := \left( x_j'\left(x\right) \right)_j,$ is injective and continuous. Moreover, the map $S:\bigwedge \longrightarrow E,\ S\left(\left(\alpha_j\right)_j \right) := \sum_{ j = 1 }^{\infty} \alpha_j x_j,$ is linear and  continuous, since
\begin{equation*}
 p \left(S\left(\left(\alpha_j\right)_j \right)\right) = p \left( \sum_{ j = 1 }^{\infty} \alpha_j x_j \right) \leq \sup_n p \left( \sum_{ j = 1}^n \alpha_j x_j \right) = q_p \left(\left(\alpha_j\right)_j\right).
\end{equation*}
As $S \circ U = I_E$ we conclude that $U$ is an isomorphism into its range $U\left (E \right )$ and $U \circ S$ is a projection of $\bigwedge$ onto $U\left (E \right ).$

$(2) \Rightarrow (3)$ is trivial, while $(3) \Rightarrow (1)$ is consequence of Proposition \ref{prop00}.
\end{proof}

The following Corollary is a consequence of an important result of Pe{\l}czy{\'n}ski. A locally convex space is said to satisfy the bounded approximation
property if  the identity of E is the pointwise limit of an equicontinuous  net of
finite rank operators. If the locally convex space is separable, then the net can be replaced by a sequence. Pelczynski \cite{041}  (see also \cite[Theorem 2.11]{033} ) proved
that a separable Fr\'echet space has
the bounded approximation property if and only if it is isomorphic to a complemented Fr\'echet  space with a Schauder basis.

\begin{corollary}
 A Fr\'echet space $E$ admits an atomic decomposition if and only if $E$ has the bounded approximation property.
\end{corollary}
\begin{proof}
 It follows from  Theorem \ref{thrm01}
and the aforementioned result of Pelczynski \cite{041}.
\end{proof}

\par\medskip\noindent

Taskinen \cite{040} gave examples of a complemented subspace $F$ of a Fr\'echet Schwartz space $E$ with a Schauder basis, such that $F$ is nuclear and does not have a basis. By Theorem \ref{thrm01}, $F$ has an atomic decomposition. Vogt \cite{045} gave examples of nuclear (hence separable) Fr\'echet spaces $E$ which do not have the bounded approximation property. These separable Fr\'echet spaces $E$ do not admit an atomic decomposition, although by Komura-Komura's Theorem \cite[Theorem 29.8]{044} they are isomorphic to a subspace of the countable product $s^{\N}$ of copies of the space of rapidly decreasing sequence, that has a Schauder basis.

\par\medskip\noindent
To end this section we discuss perturbation results. The following result, that is needed below, can be found in \cite[page 436]{031}: \textit{Let $E$ be a complete lcs and let $T:E \to E$ be an operator with the property that there exists $p_0 \in cs(E)$ such that for all $p \in cs(E)$ there is $C_p>0$ such that $p(Tx)\leq C_p p_0(x)$  for all $x \in E$ (that is, $T$ maps a neighborhood into a bounded set) and moreover $C_{p_0}$ can be chosen strictly smaller than $1.$ Then $I-T$ is invertible (with continuous inverse on $E$).}

\begin{theorem}\label{perturb}
 Let $(\{x_j'\}, \{x_j\} )$ be an atomic decomposition of a complete lcs $E.$
\begin{description}
 \item [{\rm (1)}] If $(y_j)_j$ is a sequence in $E$  satisfying that there is $p_0 \in cs(E)$ such that for all $p \in cs(E)$ there is $C_p>0$ with:
\par\indent
(i) $\sum_{j=1}^{\infty}|x_j'(x)| p(x_j -y_j) \leq p_0(x)C_p$ for each $x\in E$ and
\par\indent
(ii) $C_{p_0}$ can be chosen strictly smaller than $1,$
\par\noindent
then, there exists $(y_j')_j$ a sequence in $E'$ such that $(\{y_j'\}, \{y_j\} )$ is an atomic decomposition for $E.$
\item [{\rm (2)}] If $(y_j')_j$ is a sequence in $E'$ satisfying that there is $p_0 \in cs(E)$ such that for all $p \in cs(E)$ there is $C_p>0$ with:
\par\indent
(i) $\sum_{j=1}^{\infty} |(x_j'-y_j')(x)| p(x_j ) \leq p_0(x)C_p$ for each $x\in E$ and
\par\indent
(ii) $C_{p_0}$ can be chosen strictly smaller than $1,$
\par\noindent
then, there exists $(y_j)_j$ a sequence in $E$ such that $(\{y_j'\}, \{y_j\})$ is an atomic decomposition for $E.$
\end{description}
\end{theorem}
\begin{proof}
 In case (1) we consider the operator $T(x)=\sum_{j=1}^{\infty} x_j'(x) (x_j -y_j).$ It is well defined as the series in absolutely convergent in $E$, hence convergent, and  $T$ is continuous as $$p(Tx)\leq  \sum_{j=1}^{\infty}|x_j'(x)| p(x_j -y_j) \leq p_0(x)C_p.$$  Now, $S:=I-T$ is invertible, therefore one can take $y_j'=(S^{-1})'(x_j')$ to conclude.
\par\medskip\noindent
In case (2) we argue in the same way with the operator $T(x)=\sum_{j=1}^{\infty} (x_j'-y_j')(x) x_j,$ and the sequence $(y_j)_j$ is given by $S^{-1}(x_j), $ $j \in \N.$
\end{proof}

Our next result should be compared with \cite[Proposition 2]{030}.

\begin{corollary}
  Let $(\{x_j'\}, \{x_j\})$ be an atomic decomposition of a complete lcs $E.$ Suppose that there exists $p_0\in cs(E)$ such that $|x'_j(x)|\leq p_0(x)$ for every $x\in E, j\in \N.$ Let $\left(y_j\right)_j\subset E$ such that $\sum_{j=1}^\infty p(y_j - x_j) < \infty$ for every $p\in cs(E)$ and $\sum_{j=1}^\infty p_0(y_j - x_j) < 1.$ Then there exists $(y_j')_j\subset E'$ such that $(\{y_j'\}, \{y_j\} )$ is an atomic decomposition for $E.$
\end{corollary}

\begin{corollary}\label{cor:perturbadualfrechet}
 Let $E$ be a Fr\'echet space with fundamental system of seminorms $(p_k)_k$ and let $(\{x_j'\}, \{x_j\} )$ be an atomic decomposition of $E.$ Suppose that  $(y_j')_j\subset E'$ satisfies
$$
p_1^\ast (x_j'-y_j')<\frac{1}{1+j^2 p_j(x_j)+3^j p_1(x_j)}\ \mbox{where}\ p_1^\ast(x')=\sup \{|x'(x)|: p_1(x)\leq 1 \}.
$$ Then there exists $(y_j)_j\subset E$ such that $(\{y_j'\}, \{y_j\} )$ is an atomic decomposition for $E.$
\end{corollary}

\par\medskip\noindent
Given an atomic decomposition $(\{x_j'\}, \{x_j\} )$ on a complete lcs $E,$ if $x_1'(x_1)\neq 1$ the map $x \to \sum_{j=2}^\infty x_j'(x)x_j$ is invertible as $1$ is not an eigenvalue of the rank one operator $x \to  x_1'(x)x_1$; see \cite[p. 207]{047}. Hence there exists $(y_j')_j\subset E'$ such that $(\{y_j'\}_j, \{x_{j+1}\}_j)$ is an atomic decomposition and similarly there exists $(y_j)_j\subset E$ such that $(\{x'_{j+1}\}_j, \{y_j\}_j)$ is an atomic decomposition. That is, we can remove an element and still obtain atomic decompositions. We recall that for a Schauder basis $(x_j)_j$ with functional coefficients $(x'_j)_j$ one has $x_1'(x_1) = 1.$

\section{Duality of atomic decompositions}

Given an atomic decomposition $ \left( \{ x_j' \} , \{ x_j \} \right) $ of $E$ it is rather natural to ask whether $ \left( \{ x_j \} , \{ x_j' \} \right) $ is an atomic decomposition of $E'.$ This is always the case when $E'$ is endowed with the weak* topology $\sigma(E',E)$.

 \begin{lemma} \label{lem05}
If $ \left( \{ x_j' \} , \{ x_j \} \right) $ is an atomic decomposition of $E$, then $ \left( \{ x_j \} , \{ x_j' \} \right) $ is an atomic decomposition of $\left(E',\sigma\left(E',E\right)\right)$.
\end{lemma}
\begin{proof}
For every $x' \in E'$ and $x \in E$ we have
     \begin{equation*}
        x'\left(x\right) = x'\left( \sum_{ j = 1}^{\infty} x_j'\left(x\right)x_j \right) = \sum_{j = 1}^{\infty} x_j'\left(x\right)x'\left(x_j\right) = \left(\sum_{j=1}^{\infty} x'\left(x_j\right)x_j' \right)\left(x\right),
    \end{equation*}
and $x' = \sum_{j=1}^{\infty} x'\left(x_j\right)x_j'$ with convergence in $\left(E', \sigma\left(E',E\right)\right)$.
\end{proof}

We investigate conditions to ensure that $ \left( \{ x_j \} , \{ x_j' \} \right) $ is an atomic decomposition of the strong dual $\left(E',\beta\left(E',E\right)\right)$ of $E$.  Moreover we investigate the relation between the existence of certain  atomic decompositions and reflexivity. We recall that in the  case of bases this questions lead to the concept of  shrinking basis and boundedly complete basis; see \cite{006}.

Given an atomic decomposition $\left( \{ x_j' \} , \{ x_j \} \right)$ of a lcs $E$ we denote,  for each $n \in \N$,
$\begin{displaystyle}
T_n\left(x\right) := \sum_{ j = n+1 }^{\infty} x_j'\left(x\right) x_j,\end{displaystyle}$
that is a continuous linear operator on $E$.

\begin{definition}\label{prop04}
\begin{enumerate}\item An atomic decomposition $\left( \{ x_j' \} , \{ x_j \} \right)$ of a lcs $E$ is said to be \textit{shrinking} if, for all $ x' \in E',$ $$\lim_{n \rightarrow \infty} x' \circ T_n = 0$$ uniformly on the bounded subsets of $E.$

\item An atomic decomposition $ \left( \{ x_j' \} , \{ x_j \} \right) $ of  a lcs $E$ is said to be \textit{boundedly complete} if the series $\sum_{ j = 1}^{\infty} x'' \left(x_j'\right) x_j $ converges in $E$ for every $x'' \in E''.$
\end{enumerate}
\end{definition}

\begin{proposition}\label{prop03}
Let $E$ be a lcs and let  $\left( \{ x_j' \} , \{ x_j \} \right) $ be an atomic decomposition of $E$. The following are equivalent:
\begin{itemize}
\item[\rm (1)] $ \left( \{ x_j \} , \{ x_j' \} \right) $ is an atomic decomposition for $E_{\beta}'$.
\item[\rm (2)] For all $x' \in E'$, $ \sum_{ j = 1 }^{\infty} x'\left(x_j\right) x_j'$ is convergent in $E_{\beta}'$.
\item[\rm (3)] $ \left( \{ x_j' \} , \{ x_j \} \right) $ is shrinking.
\end{itemize}

Moreover, if $ \left( \{ x_j' \} , \{ x_j \} \right) $  is a shrinking atomic decomposition of $E$, then $ \left( \{ x_j \} , \{ x_j' \} \right) $ is a boundedly complete atomic decomposition of $E'_\beta$.
\end{proposition}
\begin{proof}
$(1) \Rightarrow (2)$ is clear by the definition of atomic decomposition.

 $(2) \Rightarrow (3)$  From the assumption and  lemma \ref{lem05}, $x'=\sum_{ j = 1 }^{\infty} x'\left(x_j\right) x_j'$  in  the topology $\beta\left(E',E\right).$ As $x'\circ T_n =\sum_{ j = n+ 1 }^{\infty} x'\left(x_j\right) x_j'$ we conclude.

Finally, we prove $(3) \Rightarrow (1).$ Every $x'\in E'$ can be written as $x'=\sum_{ j = 1 }^{\infty} x'\left(x_j\right) x_j'$ with convergence in the weak* topology $\sigma\left(E',E\right).$ Given a bounded set $B$ in $E,$ $${\rm sup}_{x \in B}\left|\left( x'-\sum_{ j = 1 }^{n} x'\left(x_j\right) x_j'\right) (x )\right|= {\rm sup}_{x \in B}|x'\circ T_n (x)|$$
which tends to zero, hence  $x'=\sum_{ j = 1 }^{\infty} x'\left(x_j\right) x_j'$  in  the topology $\beta\left(E',E\right).$

Finally, if  $ \left( \{ x_j' \} , \{ x_j \} \right) $ is a shrinking atomic decomposition of $E$, then $ \left( \{ x_j \} , \{ x_j' \} \right) $ is an atomic decomposition of $E'_\beta$. Moreover, given $x''' \in E'''$ set $x':=x'''|_E$ to obtain
     \begin{equation*}
        \sum_{ j = 1}^{\infty} x'''\left(x_j\right) x_j' =  \sum_{ j = 1}^{\infty} \left(\left.x'''\right|_E\right)\left(x_j\right) x_j'  = \sum_{ j = 1}^{\infty} x'\left(x_j\right) x_j' = x'.
    \end{equation*}

    \end{proof}

\par\medskip\noindent
Recall that a boundedly complete Schauder basis $\left (e_j\right )_j$ in a lcs $E$ is a  basis such that if $\left (\alpha_j\right )_j \in \omega$ and $ \left( \sum_{ j = 1 }^{k} \alpha_j e_j \right)_k $ is bounded,  then $\sum_{ j = 1 }^{\infty} \alpha_j e_j $ is convergent.
\par\medskip\noindent
In \cite{005} it is shown that a basis $\left (e_j\right )_j$ in a Banach space  $X$ is boundedly complete if and only if the atomic decomposition $ \left( \{ e_j' \} , \{ e_j \} \right) $ is boundedly complete. This extends to arbitrary barrelled spaces.

\begin{proposition}\label{prop05}
Let $E$ be a barrelled lcs with a Schauder basis $\left( e_j \right)_j $. Then the following are equivalent:
\begin{itemize}
\item[\rm (1)] The basis is boundedly complete.
\item[\rm (2)] The atomic decomposition $ \left( \{ e_j' \} , \{ e_j \} \right) $ is boundedly complete.
\end{itemize}
\end{proposition}
\begin{proof}
   To prove $(1) \Rightarrow (2)$ we fix $x'' \in E'' $ and we prove that $ \sum_{ j = 1 }^{\infty} e_j'\left(x''\right) e_j$ converges in $E$. For every $x'\in E'$ and $x\in E$ we have
$$
\lim_{k\to \infty}\left( \sum_{ j = 1}^k x'\left( e_j\right)e_j' \right)\left(x\right) = \lim_{k\to \infty}x'\left( \sum_{ j = 1}^k e_j'\left(x\right)e_j \right) = x'(x).$$ Since $E$ is barrelled we conclude that $\left\{ \sum_{ j = 1 }^k x'\left(e_j\right) e_j', k \in \N \right\}$ is $\beta(E',E)$-bounded. Consequently $\left\{  \sum_{ j = 1 }^{k} x''\left(e_j'\right)x'(e_j), k \in \N \right\}$ is a bounded set of scalars for every $x'\in E',$ which means that $\left\{  \sum_{ j = 1 }^{k} x''\left(e_j'\right)e_j, k \in \N \right\}$ is $\sigma(E,E')$-bounded. As all topologies of the same dual pair have the same bounded sets (\cite[8.3.4]{006}) we finally obtain that $\left\{  \sum_{ j = 1 }^{k} x''\left(e_j'\right)e_j, k \in \N \right\}$ is a bounded subset of $E$ and the conclusion follows.
\par\medskip

   To prove $(2) \Rightarrow (1)$ we fix $\left(\alpha_j\right)_j \subset \K $ such that $\left( \sum_{j = 1}^k \alpha_j e_j \right)_k $ is bounded and we show that $ \sum_{ j = 1}^{\infty} \alpha_j e_j $ is convergent in $E$. Since $\left( \sum_{j = 1}^k \alpha_j e_j \right)_k $ is $\sigma\left(E'', E'\right)$-relatively compact then it has a $\sigma\left(E'', E'\right)$-cluster point $x'' \in E''.$ By hypothesis, $\sum_{j = 1}^{\infty} x''\left( e_j' \right) e_j $ is convergent in $E$, so to conclude it suffices to check that $x''\left(e_j'\right) = \alpha_j.$ To this end we fix $j \in \N$ and $k > j$ and observe that
      \begin{equation*}
           e_j'\left( \sum_{i = 1}^k \alpha_i e_i \right) = \sum_{ i= 1}^k \alpha_i e_j'\left(e_i\right) = \alpha_j.
   \end{equation*}
   As $x''(e_j)$ is a cluster point of $\left\{e_j'\left( \sum_{i = 1}^k \alpha_i e_i \right)\right\}_{k=1}^\infty$ we finally deduce $x''\left(e_j'\right) = \alpha_j.$ \end{proof}

\begin{remark}\label{prop005}
Let $ \left( \{ x_j' \} , \{ x_j \} \right) $ be an atomic decomposition of $E$ the let $P: E \rightarrow E$ be a continuous linear projection. It is easy to see that if  $ \left( \{ x_j' \} , \{ x_j \} \right) $ is  shrinking (boundedly complete) then $ \left( \{ P'(x_j') \} , \{ P\left(x_j\right) \} \right) $ is a shrinking (boundedly complete) atomic decomposition for $P\left(E\right)$.
\end{remark}

\begin{lemma}\label{lem07}
Suppose that $ \left( \{ x_j' \} , \{ x_j \} \right) $ is an atomic decomposition of a barrelled lcs $E$ such that for all $k \in \N$
\begin{equation}\label{eq06}
\lim_{ n \rightarrow \infty} \left(x_k' - \sum_{ j = 1}^n x_k'\left(x_j\right) x_j'\right) = 0 \mbox{ in } E'_\beta.
\end{equation}
Then $ \left( \{ x_j \} , \{ x_j' \} \right) $ is an atomic decomposition of the closed linear span $H = \overline{{\rm span}\left\{x_j'\right\}}^{E'_\beta}$.
\end{lemma}
\begin{proof}
We fix $x' \in H$ and show that $x' = \sum_{ j = 1}^{\infty} x'\left(x_j\right)x_j'$ with convergence in $E'_{\beta}$. To this end we fix $U$ a neighborhood of zero in $E'_{\beta}$ and consider $F_n(x) = \sum_{j=1}^nx_j'(x)x_j,$ $n\in \N, x\in E$. Since $\left(F_n'\right)_n \subset L\left(E'\right)$ is equicontinuous, there is another $\beta\left(E',E\right)$-neighborhood $V,$ $V \subset U$, such that $F_n'\left(V\right) \subset \frac{1}{3} U$ for each $n \in \N$. Find $u = \sum_{ k = 1}^{s} \alpha_k x_k', \alpha_k \in \K, s \in \N,$ with $x' - u \in \frac{1}{3}V$. By condition $\left(\ref{eq06}\right)$ we can find $n_0 \in \N$ such that $u - F_n'\left(u\right) \in \frac{1}{3}V$ for each $n \geq n_0$. Finally,
\begin{equation*}
x' - F_n'\left(x'\right) = x' - u - F_n'\left(x' - u\right) + u - F_n'\left(u\right) \in \dfrac{1}{3}V + \dfrac{1}{3}U + \dfrac{1}{3}V \subset U \mbox{ if } n \geq n_0.
\end{equation*}
Thus $\begin{displaystyle}E'_\beta\mbox{-}\lim_{n \rightarrow \infty} F_n'\left(x'\right) = x'\end{displaystyle}$ and the conclusion follows.
\end{proof}

\begin{remark}
\begin{description}
  \item[{\rm (a)}] Observe that if $ \left( \{ x_j \} , \{ x_j' \} \right) $ is an atomic decomposition of the closed linear span $H = \overline{{\rm span}\left\{x_j'\right\}}^{E'_{\beta}}$ then $\left(\ref{eq06}\right)$ holds since $x_k' \in H$ for each $k \in \N$.
  \item[{\rm (b)}] If $\{x_j \}$ is a Schauder basis in $E$ with functional coefficients $\{x_j' \}$ then $\left(\ref{eq06}\right)$ also holds, since $x_k'- \sum_{j =1}^n x_k'\left(x_j\right)x_j' = 0$ for every $n \geq k$.
  \item[{\rm (c)}] If $E$ is a Montel space, $\left(\ref{eq06}\right)$ holds since every weakly convergent sequence in a Montel space is also strongly convergent to the same limit, by \cite[11.6.2]{006}.
\end{description}
\end{remark}

\begin{theorem}\label{prop07} Let $\left( \{ x_j' \} , \{ x_j \} \right)$ be an atomic decomposition of a lcs $E.$ Then,
\begin{itemize}
\item[\rm (1)] If $ \left( \{ x_j' \} , \{ x_j \} \right) $ is  boundedly complete atomic decomposition, $E$ is a barrelled and complete lcs $E$ with $E_{\beta}''$ barrelled, then $E$ is complemented in its bidual $E_{\beta}''$.

\item[\rm (2)]  If $E$ is reflexive and  $\left(\ref{eq06}\right)$ in Lemma \ref{lem07} holds, then $ \left( \{ x_j' \} , \{ x_j \} \right) $ is shrinking.

\item[\rm (3)] If $ \left( \{ x_j' \} , \{ x_j \} \right) $ is shrinking and boundedly complete, then $E$ is semi-reflexive. If, in addition, $E$ is barrelled then it is reflexive.
\end{itemize}
\end{theorem}

\begin{proof}

(1) Since $ \left( \{ x_j' \} , \{ x_j \} \right) $ is  boundedly complete the linear map $P:E'' \to E,\ P\left(x''\right) := \sum_{ j = 1 }^{\infty} x''\left(x_j'\right) x_j$ is well defined. Since $E''_\beta$ is barrelled we can apply Banach-Steinhaus theorem to conclude that $P$ is continuous, and it is clearly surjective. As $E$ is barrelled, it can be canonically identified with a topological subspace of its bidual $E_{\beta}''$. Then it is easy to see that $P$ is a projection.

(2) As $E$ is reflexive then it is barrelled and Lemmas \ref{lem05} and \ref{lem07} hold. In particular, for each $x' \in H = \overline{{\rm span}\{x_j'\}}^{E'_\beta}$ we have $x' = \sum_{j = 1}^{\infty} x'\left(x_j\right) x_j'$ with convergence in $E'_\beta$. Since $E$ is semi-reflexive, $\beta\left(E',E\right)$ and $\sigma\left(E',E\right)$ are topologies of the same dual pair. Hence, by Lemma $\ref{lem05}$ we obtain $H = \overline{{\rm span}\{x_j'\}}^{E'_\beta} = \overline{{\rm span}\{x_j'\}}^{\left(E', \sigma\left(E',E\right)\right)} = E'$. The result follows by Proposition \ref{prop03}.

(3) Fix $x'' \in E''$. Since the atomic decomposition is boundedly complete then $\sum_{ j = 1}^{\infty} x_j'\left(x''\right)x_j $ converges to an element $x\in E$. We claim that $x'' = x.$ In fact, since the atomic decomposition is shrinking, for every $x'\in E'$ we have $x' = \sum_{ j =1}^{\infty} x'\left(x_j\right) x_j'$ with convergence in $E'_\beta$. Thus
\begin{equation*}
\langle x'', x' \rangle = \langle x'', \sum_{ j =1}^{\infty} x'\left(x_j\right) x_j' \rangle = \sum_{ j =1}^{\infty} x'\left(x_j\right)x''\left(x_j'\right) = \left( \sum_{j = 1}^{\infty} x''\left(x_j'\right)x_j\right)\left(x'\right) = \langle x, x'\rangle.
\end{equation*} It follows $x'' = x.$
\end{proof}

For a Fr\'echet space $E$, the bidual $E_{\beta}''$ is again a Fr\'echet space, therefore barrelled. For LB-spaces, this is not always the case. In fact, if $\lambda_1(A)$ is the Grothendieck example of a non-distinguished Fr\'echet space, $\lambda_1(A)$ is the strong dual of an LB-space $E.$ The bidual of $E$, being the strong dual of $\lambda_1(A),$ is not barrelled. See \cite[Chapter 31, Sections 6 and 7]{018} and \cite[Example 27.19]{044}.

\section{Unconditional atomic decompositions} \label{uncond}

In this section we assume that $E$ is a complete lcs and we denote by $\mathcal{U}_0(E)$ the set of absolutely convex and closed 0-neighborhoods. We refer the reader to \cite{006} for unconditional convergence of series in locally convex spaces.

\begin{definition}
An atomic decomposition $ \left( \{ x_j' \} , \{ x_j \} \right) $ for a lcs $E$ is said to be \textit{unconditional} if for every $x \in E$ we have $ x =\sum_{ j = 1 }^{\infty} x_j'\left(x\right)x_j$ with unconditional convergence.
\end{definition}

\begin{remark}\label{remark1}
By  \cite[p.116]{011} a series $\sum_{j = 1}^{\infty} x_j$  in a (sequentially) complete lcs converges unconditionally if and only if the limits  $\begin{displaystyle}\lim_{n\to \infty} \sum_{j = 1}^{n} a_j x_j\end{displaystyle}$ exist uniformly for $\left(a_j \right)_j $ in the unit ball of $\ell_{\infty},$  and the operator $\ell_{\infty} \to E,\ \left\{ a_j \right\} \mapsto \sum_{ j = 1 }^{\infty} a_j x_j,$ is continuous.
\end{remark}

\begin{lemma}\label{prop09}
Let $X$ be a normed space, $E$ a barrelled space and $G$ any lcs. Then every separately continuous bilinear map $B: X \times E \rightarrow G$ is continuous.
\end{lemma}
\begin{proof}
Let $W\in U_0(G)$ and let $U_X$ be the closed unit ball of $X$. Now we take $T:=\left\{y\in E:\ B\left(x,y\right)\in W \right.$ for every $\left.x\in U_X \right\} = \bigcap_{ x \in U_X } B_x^{-1}\left(W\right),$ where $B_x = B(x,\cdot).$ Note that $T$ is an absolutely convex closed subset since each $B_x:E \rightarrow G$ is continuous. Fixing $y \in E$, since $B_y : X \rightarrow G$ is continuous then $B_y^{-1} \left(W\right) \in \mathcal{U}_0 \left(X\right),$ what means that there exists $\lambda > 0$ such that $\lambda U_X \subset B_y^{-1} \left(W\right)$. Therefore $B\left(x, \lambda y \right) \in W $ for every $x \in U_X$ and $\lambda y\in T$, that is, $T$ is absorbent. Since $E$ is barrelled then $T\in U_0\left(E\right)$ and from $B\left( U_X \times T\right)\subset W$ we conclude that $B$ is continuous.
\end{proof}

\begin{corollary}\label{coro31}
Let $ \left( \{ x_j' \} , \{ x_j \} \right) $ be an unconditional atomic decomposition for a complete  barrelled lcs $E$. Then, the bilinear map
$$
B:E \times \ell_{\infty}\to E,\  B\left(x, a\right):= \sum_{ j = 1 }^{\infty}a_j x_j'\left(x\right)x_j,
$$ is continuous.
\end{corollary}

The property of having an unconditional atomic decomposition is also inherited by complemented subspaces.

\begin{proposition} \label{prop10}
Let $E$ be a lcs and let $P:E \rightarrow E$ be a continuous linear projection. If $\left( \{ x_j' \} , \{ x_j \} \right)$ is an unconditional atomic decomposition for $E$, then $\left( \{ P'(x_j') \} , \{ P\left(x_j\right) \} \right)$ is an unconditional atomic decomposition for $P\left(E\right)$.

In particular if $E$ is isomorphic to a complemented subspace of a lcs with a unconditional Schauder basis, then $E$ admits an unconditional atomic decomposition.
\end{proposition}

\par\medskip\noindent
Similarly to Lemma \ref{lem02} we have the following.

\begin{lemma}\label{lem11}
    Let $\left( x_j \right)_j$ be a fixed sequence of non-zero elements in a lcs $E$ and let us denote by $\widetilde{\bigwedge}$ the space
    \begin{equation}\label{eq03}
            \widetilde{\bigwedge} := \{ \alpha = \left(\alpha_j\right)_j \in \omega : \sum_{j = 1}^{\infty} \alpha_j x_j \mbox{ is unconditionally convergent in } E \}.
        \end{equation}
 Endowed with the system of seminorms
        \begin{equation}\label{eq04}
            \widetilde{{\mathcal Q}} := \left\{ \widetilde{q}_p\left(\left(\alpha_j\right)_j\right) := \sup_{b \in B_{\ell^{\infty}}} p\left(\sum_{j = 1}^{\infty} b_j \alpha_j x_j \right) , \mbox{ for all } p \in cs(E) \right\},
        \end{equation}
$\widetilde{\bigwedge}$ is a complete lcs of sequences and the canonical unit vectors are an unconditional basis.
\end{lemma}

\begin{theorem}\label{thrm11}
Let $E$ be a complete, barrelled lcs. The following conditions are equivalent:
\begin{itemize}
\item[\rm (1)] $E$ admits an unconditional atomic decomposition.
\item[\rm (2)] $E$ is isomorphic to a complemented subspace of a complete sequence space with the canonical unit vectors as unconditional Schauder basis.
\item[\rm (3)] $E$ is isomorphic to a complemented subspace of a complete sequence space with unconditional Schauder basis.
\end{itemize}
\end{theorem}
\begin{proof}
The proof follows the steps of Theorem \ref{thrm01} but the continuity of the map $$U:E\longrightarrow \widetilde{\bigwedge}, \,\,x \to \left(x'_j(x)\right )_j,$$ now follows from Corollary \ref{coro31}.
 \end{proof}

In our next two results, bipolars are taken in $E''$ that is $U^{\circ \circ}=\{x''\in E'': |x''(x')|\leq 1 \mbox{ for all } x' \in U^{\circ } \}.$

\begin{lemma}\label{lem30}
Let $E$ be a lcs and let $U$ be an absolutely convex and closed 0-neighborhood. For every $ z \in E''$ such that $p_{U^{\circ \circ}}\left(z\right) > 0$ there exists $\left(x_{\alpha}\right) \subset E$ with $p_U\left(x_{\alpha}\right) \leq p_{U^{\circ \circ}} \left(z\right)$ and $x_{\alpha} \to z$ in $\sigma\left(E'', E'\right)$.
\end{lemma}
\begin{proof}
First, we observe that $x:= \frac{z}{p_{U^{\circ \circ}}(z)}\in U^{\circ \circ}$, a set that coincides with  $\overline{U}^{\sigma\left(E'', E'\right)}$ by the bipolar Theorem (\cite[8.2.2]{006}). Therefore there exists $\left(y_{\alpha}\right)_{\alpha} \subset U$ such that $y_{\alpha} \to x$ in $\sigma\left(E'', E'\right)$. Now, it suffices to take $x_{\alpha} := p_{U^{\circ \circ}}\left(z\right)y_{\alpha}.$
\end{proof}
\begin{theorem}\label{thrm21}
Let $E$ be a complete, barrelled lcs which admits an unconditional atomic decomposition $ \left( \{ x_j' \} , \{ x_j \} \right) $. Then, $ \left( \{ x_j' \} , \{ x_j \} \right) $ is boundedly complete if and only if $E$ does not contain a copy of $c_0$.
\end{theorem}

\begin{proof}
Suppose that $E$ contains a copy of $c_0$. Since $E$ is separable, there exists a projection $P: E \rightarrow E$ such that $P\left(E\right)$ is isomorphic to $c_0$ (\cite[8.5.9]{006}). If $ \left( \{ x_j' \} , \{ x_j \} \right) $ is boundedly complete, then $ \left( \{ P'(x_j') \} , \{ P\left(x_j\right) \} \right) $ is a boundedly complete atomic decomposition in $P\left(E\right)\simeq c_0$. By Proposition \ref{prop07}, $c_0$ is complemented in its bidual, a contradiction.

In order to show the converse, suppose that $E$ does not contain a copy of $c_0$ and $\left( \{ x_j' \} , \{ x_j \} \right)$ is not boundedly complete. Then there exists $x'' \in E'',$ $x''\neq 0,$ such that $\sum_{ j = 1}^{\infty} x''\left(x_j'\right)x_j$ is not  convergent in $E$. We can find an absolutely convex 0-neighborhood $U_1$ and two sequences $\left(p_i\right)$, $\left(q_i\right)$ of natural numbers such that $p_1 < q_1 < p_2 < q_2 < \ldots $ and $p_{U_1} \left( \sum_{ i = p_j }^{q_j} x''\left(x_i'\right)x_i \right) \geq 1$ for each $j \in \N$. We set $y_j := \sum_{ i = p_j}^{q_j} x''\left(x_i'\right)x_i$ and define $T: \varphi \rightarrow E$ by $T\left(\left(a_j\right)_j\right) := \sum_{j = 1}^{\infty} a_j y_j$. We first prove that $T$ is continuous when $\varphi$ is endowed with the $\left\| \cdot  \right\|_{\infty}$- norm. To this end, take $U$ an absolutely convex neighborhood of the origin in $E.$ Since $x'' \neq 0$, $x'' \in E''$, there is an absolutely convex  0-neighborhood $U_2$ in $E$ such that $p_{U_2^{\circ \circ}}\left( x''
 \right) >0 $.  Put $V := U_1 \cap U_2 \cap U$.  Clearly $p_{V^{\circ \circ}}\left(x''\right) \geq p_{U_2^{\circ \circ}}\left(x''\right) > 0$.  We can apply Corollary \ref{coro31} to find an absolutely convex closed 0-neighborhood $W$ in $E$ such that $W \subset V$ and
    \begin{equation}\label{eq07}
       p_V\left( \sum_{ i = 1}^{\infty} d_i x_i'\left(z\right) x_i \right) \leq p_{W}\left(z\right)\left\|d\right\|_{\infty}
    \end{equation}
for each $n \in \N$, each $d \in \ell^{\infty}$ and $z \in E$. For $a = \left(a_j\right)_j \in \varphi$, and $s:={\rm max( supp}\, a{\rm )}$, the support of $a$ being the set of non-zero coordinates of $a$,  we define $b_i = a_j$ for $p_j \leq i \leq q_j$, and $b_i = 0$ otherwise. We have
    \begin{equation*}
     \sum_{ j = 1}^{\infty} a_j y_j = \sum_{ j = 1}^{s} a_j y_j = \sum_{ i =p_1}^{q_{s}} b_i x''\left(x_i'\right)x_i.
    \end{equation*}

Given $\varepsilon > 0$,  we can apply Lemma \ref{lem30} to find $y \in E$, $p_{W}\left(y\right) \leq p_{W^{\circ \circ}}\left(x''\right)$ and
    \begin{equation*}
      \max_{p_1 \leq i \leq q_{s}} \left| \left(x'' - y \right)\left(x_i'\right) \right| \leq \frac{ \varepsilon}{ 2q_s\left\| a \right\|_{\infty} \max\left(p_V\left(x_i\right), 1\right)}.
    \end{equation*}
This implies
    \begin{eqnarray*}
      p_V\left( \sum_{ i = p_1}^{q_{s}} b_i x''\left(x_i'\right)x_i \right) & \leq & p_V\left( \sum_{ i = p_1 }^{q_{s}} b_i x_i'\left(y\right)x_i \right) + \sum_{ i = p_1}^{q_s} \left| b_i \right| \left| \left(x''- y\right)\left(x_i'\right)\right| p_V\left(x_i\right) \leq \nonumber \\
      & \leq & p_V\left( \sum_{ i = p_1}^{q_{s}} b_i x_i'\left(y\right)x_i \right) + \frac{\varepsilon}{2}.
    \end{eqnarray*}
 Now, by the estimate (\ref{eq07}), we obtain
    \begin{equation*}
      p_V\left(\sum_{ i = p_1}^{q_{s}} b_i x_i'\left(y\right)x_i \right) \leq \left( \max_{p_1 \leq i \leq q_{s}} \left|b_i\right| \right)p_{W}\left(y\right) \leq \left( \max_{ 1 \leq j \leq s} \left| a_j \right| \right)p_{W^{\circ \circ}}\left(x''\right).
    \end{equation*}

Then,

      $$p_V\left( \sum_{ j = 1}^{s} a_j y_j \right)  \leq \left\| a\right\|_{\infty} p_{W^{\circ \circ}}\left(x''\right) + \dfrac{\varepsilon}{2}.
   $$
Since this holds for each $\varepsilon > 0$ , we get $$p_{U}\left( \sum_{ j = 1}^{\infty} a_j y_j \right)\leq p_V\left( \sum_{j = 1}^\infty a_j y_j \right) \leq  \left\| a \right\|_{\infty} p_{W^{\circ \circ}}\left(x''\right).$$

 Thus the operator $T: \left(\varphi, \left\|\cdot\right\|_{\infty} \right) \rightarrow E$ is  continuous. Since $E$ is complete, $T$  admits a unique continuous extension $\widetilde{T}: c_0 \rightarrow E$. As by assumption $E$ does not contain $c_0$, we can apply Theorem 4 in \cite[p.208]{023} to conclude that $\left( \widetilde{T}\left(e_j\right) \right)_j$ has a convergent subsequence $\left(\widetilde{T}\left(e_{j_k}\right)\right)_k.$ That is, $\left ( y_j \right )_j$ admits a convergent subsequence $\left(y_{j_k}\right)_k$. Moreover, since $\widetilde{T}: \left(c_0, \sigma\left(c_0, l_1\right)\right) \rightarrow \left(E, \sigma\left(E,E'\right)\right)$ is also continuous then $\left (\widetilde{T}\left(e_j\right)\right)_j = \left (y_j\right)_j$ is $\sigma\left(E,E'\right)$-convergent to 0, hence $\left(y_{j_k}\right)_k$ must converge to 0 in $E.$ This is a contradiction, since $p_{U_1}\left(y_j\right) \geq 1$ for each $j \in \N$.
\end{proof}

\begin{definition}{\rm \cite{010}}
An $\left( LF \right)$-space $E = \IL_{n \to } E_n $ is called \textit{boundedly retractive} if for every bounded set $B$ in E  there exists $m=m(B) $ such that $B$ is contained and bounded in  $E_m$ and $E_m$ and $E$ induce the same topology on $B.$ \end{definition}
\par\medskip\noindent
By \cite{029} an $\left( LF \right)$-space $E$ is boundedly retractive if and only if each bounded subset in $E$ is in fact bounded in some step $E_n$ and  for each $n$ there is $m>n$ such that $E_m$ and $E$ induce the same topology on the bounded sets of $E_n.$
\par\medskip\noindent
For $\left( LB \right)$-spaces,  this is equivalent to the  a priori weaker condition that for all  $n \in \N$, there exists $m > n$ such that  for all $k>m$, $E_m$ and $E_k$ induce the same topology in the unit ball $B_n$  of $E_n$ (\cite{027}). In particular  $\left( LB \right)$-spaces with compact linking maps $E_n \hookrightarrow E_{n+1}$ are boundedly retractive. More information about these and related concepts can be seen in \cite{042}.

Obviously, each Fr\'echet space $F$ can be seen as a boundedly retractive $\left( LF \right)$-space, just take $F_n=F$ for all $n \in \N.$ In particular \ref{boundedlyLF} holds for Fr\'echet spaces. Every strict (LF)-space is boundedly retractive. In particular, for a open subset $\Omega$  in $\R^d ,$ the space ${\mathcal D}(\Omega)$ is a boundedly retractive $\left( LF \right)$-space. The space ${\mathcal E}'(\Omega)$  and  the space $HV$ in Example 1 of Section \ref{Examples} are  boundedly retractive $\left( LB \right)$-spaces.

\par\medskip\noindent
Rosenthal $\ell_1$-theorem was extended to Fr\'echet spaces by D\'{\i}az in \cite{030}, showing that every bounded sequence in a Fr\'echet space has a subsequence that is either weakly Cauchy or equivalent to the unit vectors in $\ell_1.$

\begin{proposition}\label{prop30}
{\rm (}Rosenthal $\ell_1$-theorem for $\left(LF\right)$-spaces{\rm )} Let $E = \IL_{n \to } E_n $ be a boundedly retractive $\left(LF\right)$-space. Every bounded sequence in $E$ has a subsequence which is $\sigma\left(E,E'\right)$-Cauchy or equivalent to the unit vector basis of $\ell_1$. In particular, $E$ does not contain a copy of $\ell_1$ if and only if every bounded sequence in $E$ has a $\sigma\left(E,E'\right)$-Cauchy subsequence.
\end{proposition}
\begin{proof}
Let $\left(x_j\right)_j$ be a bounded sequence in $E$ and assume that has no $\sigma\left(E, E'\right)$-Cauchy subsequence. There is $n_0 \in \N$ such that $\left(x_j\right)_j$ is a bounded sequence in $E_{n_0}.$ Now select $m \geq n_0$ such that $E_m$ and $E$ induce the same topology on the bounded sets of $E_{n_0}.$  Since $\left(x_j\right)_j$ is bounded in $E_m$ and it has no $\sigma\left(E_m, E_m'\right)$-Cauchy subsequence, we can apply Rosenthal's $\ell_1$-Theorem in the Fr\'echet space $E_m$ to conclude that there is a subsequence  $\left(x_{j_k}\right)_k$ which is equivalent to the unit vector basis of $\ell_1$. That is, there exist $c_1$ and a continuous seminorm $p$ in $E_m$ such that
\begin{equation*}
 c_1 \sum_{ k = 1}^{\infty} \left| \alpha_k \right| \leq p\left( \sum_{k=1}^{\infty} \alpha_k x_{j_k}\right) \leq \sup_{k}p( x_{j_k} )\sum_{ k = 1}^{\infty} \left| \alpha_k \right|,
\end{equation*} for every $\alpha = \left(\alpha_k\right)_k \in \ell_1.$

As the inclusion $E_{n_0}\hookrightarrow E_m$ is continuous, we find a continuous seminorm $q$ in $E_{n_0}$ such that for  $x \in E_{n_0}$ one has $p(x)\leq q(x).$  Then, for each  $\alpha = \left(\alpha_k\right)_k \in \ell_1$,
\begin{equation*}
 c_1 \sum_{ k = 1}^{\infty} \left| \alpha_k \right| \leq p\left(\sum_{k=1}^{\infty} \alpha_k x_{j_k}\right) \leq q\left( \sum_{k=1}^{\infty} \alpha_k x_{j_k}\right) \leq \sup_{k}q( x_{j_k} ) \sum_{ k = 1}^{\infty} \left| \alpha_k \right|.
\end{equation*}
Set $F:=\left\{ \sum_{k=1}^{\infty} \alpha_k x_{j_k} : \alpha = \left(\alpha_k\right)_k \in \ell_1 \right\} \subset E_{n_0}$. Then $p$ and $q$ restricted to $F$ are equivalent norms, and $F$ endowed with any of them  is a Banach space  isomorphic to $\ell_1$. The spaces $E_{n_0}$ and $E_m$ induce on $F$ the same (Banach) topology.
 Denote by $U_F$ the closed unit ball of $F$ and by $\tau_m$ and $\tau$ the topologies of $E_m$ and $E$, respectively. Then $\tau$ and $\tau_m$ coincide on $U_F,$  which is an absolutely convex 0-neighbourhood for $\tau_m |_F$. Applying a result of Roelcke \cite[8.1.27]{010} we conclude that $\tau_m$ and $\tau$ coincide in $F$; hence, there is a continuous seminorm $r$ on $E$ such that $p( z ) \leq r\left(z\right)$ for every $z \in F$. This implies, for each $\alpha = \left( \alpha_k\right)_k \in \ell_1$,
\begin{equation*}
 c_1 \sum_{ k = 1}^{\infty} \left| \alpha_k \right| \leq p\left( \sum_{k=1}^{\infty} \alpha_k x_{j_k}\right)\leq r \left( \sum_{k=1}^{\infty} \alpha_k x_{j_k}\right) \leq \left(\sup_{k} r\left(x_{j_k}\right) \right) \sum_{ k = 1}^{\infty} \left| \alpha_k \right|.
\end{equation*}
Thus, $\left(x_{j_k}\right)_k$ is equivalent to the unit vectors of $\ell_1$ in $E$ and the inclusion $F \hookrightarrow E$ is a topological isomorphism into. Then, $E$ contains an isomorphic copy of $\ell_1$.
\end{proof}

We use the notation $\mu\left(E', E\right)$ for the topology on $E'$ of uniform convergence on the absolutely convex and $\sigma(E,E')$-compact sets. In the proof of the next result we utilize the fact that a boundedly retractive $\left( LF \right)$-space $E$ does not contain $\ell_1$ if and only if every $\mu\left(E', E\right)$-null sequence in $E'$ is $\beta\left(E', E\right)$-convergent to 0. This was proved by Doma\'nski and Drewnowski and by Valdivia independently for Fr\'echet spaces. The proof can be seen in \cite{022} and the proof for arbitrary boundedly retractive $\left( LF \right)$-spaces follows the same steps as in \cite[Theorem 10]{022} but using Proposition \ref{prop30} instead of Rosenthal $\ell_1$-theorem for Fr\'echet spaces.

\begin{theorem}\label{boundedlyLF}
Let $E$ be a boundedly retractive $\left( LF \right)$-space. Assume that $E$ admits an unconditional atomic decomposition $ \left( \{ x_j' \} , \{ x_j \} \right) $. Then, $ \left( \{ x_j' \} , \{ x_j \} \right) $ is shrinking if and only if $E$ does not contain a copy of $\ell_1$.
\end{theorem}
\begin{proof}
We first assume that $\left( \{ x_j' \} , \{ x_j \} \right)$ is shrinking. Then, by Proposition \ref{prop03}, $\left( \{ x_j' \} , \{ x_j \} \right)$ is an atomic decomposition for $E'_\beta$ and, in particular, $E'_\beta$ is separable. Consequently $E$ contains no subspace isomorphic to $\ell_1$.

Conversely, assume that $E$ does not contain a copy of $\ell_1$. By Lemma \ref{lem05}, $\left( \{ x_j \} , \{ x_j' \} \right)$ is an atomic decomposition of $\left(E',\sigma\left(E',E\right)\right)$. We check that, for all $x' \in E'$, \begin{equation}\label{eq:seriedual}\begin{displaystyle}\sum_{j=1}^{\infty} x'(x_j)x'_j\end{displaystyle}\end{equation} is subseries summable to $x'$ in $E'_\beta$. Since for each $x \in E$  the convergence of \begin{equation}\label{eq:incondicional}
\begin{displaystyle}\sum_{j=1}^{\infty}x'_j(x)x_j\end{displaystyle}\end{equation} is unconditional and $E$ is sequentially complete, then (\ref{eq:incondicional}) is subseries summable and we conclude that  (\ref{eq:seriedual}) is also $\sigma\left(E', E\right)$-subseries summable. We can apply Orlicz-Pettis' Theorem (\cite[p. 308]{006}) to obtain that (\ref{eq:seriedual}) is $\mu\left(E', E\right)$-unconditionally convergent to $x'$. Therefore it is $\beta\left(E', E\right)$-convergent to $x'$, as $E$ does not contain a copy of $\ell_1.$ Consequently $ \left( \{ x_j' \} , \{ x_j \} \right) $ is shrinking.
\end{proof}

\section{Examples}\label{Examples}

In this section we will present some examples of atomic decomposition on locally convex spaces. These  atomic decompositions are  shrinking and boundedly complete since all the spaces involved are Montel spaces.
\par\medskip\noindent
\textbf{Example 1.} This example was obtained by Taskinen in \cite{007}. Denote by $\D$ the open unit  disc $\D := \left\{ z \in \C : \left| z \right| < 1 \right\}$ and for each $n$ let $v_n$ be the weight $v_n\left(z\right):= \min\left\{ 1, \left| \log\left( 1- \left|z \right|\right)\right|^{-n}\right\}.$ We consider the weighted Banach space of holomorphic functions
 $$ H_{v_n}^{\infty} := \left\{ f: \D \rightarrow \C \mbox{ analytic } : \left\| f \right\|_{v_n} = \sup_{ z \in \D}\left|f\left(z\right)\right|v_n\left(z\right) < \infty\right\}.$$

 Since $v_{n+1}\leq v_n$ then $H_{v_n}^{\infty} \subset H_{v_{n+1}}^{\infty}$ continuously and we consider the inductive limit
$$HV = \textrm{ind}_{ n \rightarrow \infty} H_{v_n}^{\infty}.$$

The unit disc $\D$ is decomposed as $\D := \bigcup_j D_j $ with $\overset{\circ}{D}_j \neq \emptyset$ for all $j \in \N$  in such a way that the set of elements of $\D$ belonging to more that one of the  $D_j$'s has Lebesgue measure 0. Let us fix, for all $j \in \N $, $\lambda_j \in \overset{\circ}{D}_j.$ As proved in \cite{007}, we can obtain such a decomposition with the property that
$$
		  S:HV  \to  HV,\  f \mapsto \left(Sf\right)\left(z\right):= \sum_{ j = 1}^{\infty} \dfrac{ m\left(D_j\right)f\left(\lambda_j\right)}{\left(1- \overline{\lambda_j}z\right)^2},
	$$
is an isomorphism.

\begin{theorem}{\rm \cite[Theorem 1]{007}} Under the conditions above, let $u_j\left(f\right):= \left(S^{-1}f\right)\left(\lambda_j\right)$ and $f_j(z):= \frac{m(D_j)}{(1-\overline{\lambda_j}z)^2}$ be given. Then $\left( \left\{u_j\right\} , \left\{f_j\right\}\right)$ is a shrinking and boundedly complete atomic decomposition for $HV$.
\end{theorem}
\begin{proof}
Each $ f \in HV $ can be written as
\begin{equation*}
f = S\left(S^{-1}\left(f\right)\right) = \sum_{ j = 1}^{\infty} \left( S^{-1}f\right)\left(\lambda_j\right)f_j,
\end{equation*}
hence $\left( \left\{u_j\right\} , \left\{f_j\right\}\right)$ is an atomic decomposition in $HV$. Since $HV$ is a Montel space we can apply Theorem \ref{prop07} to conclude that the atomic decomposition is shrinking.
\end{proof}

\par\medskip\noindent
As pointed in \cite[p.~330]{007}, the coefficients in the series expansion above are not unique.

\vspace{.3cm}

\noindent
\textbf{Example 2.} Let $K$ be a compact subset of $\R^p$ that coincides with the closure of its interior, i.e.\ $K = \overline{\overset{\circ}{K}}$. Let $C^{\infty}\left(K\right)$ be the space of all complex-valued functions $f \in C^{\infty}(\overset{\circ}{K})$ uniformly continuous in $\overset{\circ}{K}$  together with all partial derivatives. The Fr\'echet space topology in $C^{\infty}(K)$ is defined by the norms:
\begin{equation*}
q_n \left(f\right) := \sup\left\{ \left| f^{\left(\alpha\right)}\left(x\right)\right| : x \in K, \, \left|\alpha\right| \leq n \right\}, n \in \N_0.
\end{equation*}

A continuous and linear extension operator is a continuous and linear operator $T: C^{\infty}(K) \rightarrow C^{\infty}\left(\R^p\right)$ such that $\left.T(f)\right|_K = f.$ Not every compact set admits a continuous and linear extension operator but every convex compact set does. Further information can be found in \cite{009}.

\begin{theorem}\label{th01} Let $K \subset {\mathbb R}^p$ be a compact set which is the closure of its interior. The following conditions are equivalent:
\begin{itemize}
 \item[\rm (1)] There exists a continuous and linear extension operator $T: C^{\infty}(K) \rightarrow C^{\infty}\left(\R^p\right).$
\item[\rm (2)] There are sequences $\left( \lambda_j\right)_j \subset \R^p$ and $\left(u_j\right)_j \in C^{\infty}\left(K\right)'$  such that $\left( \left\{u_j\right\}, \left\{e^{2 \pi i x \cdot \lambda_j }\right\} \right)$ is an unconditional atomic decomposition for $C^{\infty}(K)$.
\end{itemize}
\end{theorem}
\begin{proof}
$(1)\Rightarrow (2).$ We consider $M > 0$ such that $ K \subset \left[ -M, M \right]^p$ and choose $\phi \in D\left( \left[ -2M, 2M \right]^p\right)$ such that $\phi(x)=1 $ for all $x$ in a neighborhood of $\left[ -M, M \right]^p$. For every $f \in C^{\infty}(K)$ we define $Hf = \phi\left(T\left(f\right)\right) \in D\left( \left] -2M, 2M \right[^p\right)$. Then $H : C^{\infty}\left(K\right) \to D\left( \left] -2M, 2M \right[^p\right)$ is a continuous and linear map and $\left.Hf\right|_K = f$. After extending $Hf$ as a periodic $C^{\infty}$ function in $\R^p$ we get
\begin{equation*}
Hf\left(x\right) := \sum_{ j \in \Z^p }a_j e^{2 \pi i x \cdot \lambda_j} \mbox{, where } \lambda_j = \dfrac{1}{4M}\left(j_1 , \ldots, j_p\right)
\end{equation*}
 and $a_k = a_k\left(Hf\right)$ are the Fourier coefficients of $Hf$. By \cite{008}, $\sup_{ j \in \Z^p} \left| a_j\right|\left|j\right|^m < \infty$ for every $m$, which implies that the series $f = \sum_{ j \in \Z^p }a_j e^{2 \pi i x \cdot \lambda_j}$ converges absolutely in $ C^{\infty}\left(K\right).$ Each $a_k,$ being a Fourier coefficient of $Hf,$ depends linearly and continuously on $f.$ Then $\left( u_j\left(\cdot\right),e^{2 \pi i x \cdot \lambda_j}\right)_{ j \in \Z^p}$ is an atomic decomposition for $ C^{\infty}\left(K\right)$, with $u_j \in C^{\infty}\left(K\right)'$  defined by $u_j\left(f\right) = a_j\left(Hf\right)$.
\par
$(2)\Rightarrow (1).$ For every $f\in C^{\infty}(K)$ we have
$$
f(x) = \sum_{j=1}^{\infty}u_j(f) e^{2 \pi i x \cdot \lambda_j}\ \mbox{in}\ C^{\infty}(K)$$ and
$$
\sum_{j=1}^{\infty}u_j(f)b_j e^{2 \pi i x \cdot \lambda_j}$$ converges in $C^{\infty}(K)$ for every $(b_j)\in \ell_{\infty}.$ After differentiation, we obtain that the series
$$
\sum_{j=1}^{\infty}u_j(f)2 \pi b_j\lambda_j^\alpha e^{2\pi i x \cdot \lambda_j }$$ converges in $C^{\infty}(K)$ for every $\alpha \in {\mathbb N}_0^p$ and $(b_j)\in \ell_{\infty}.$ In particular, this series converges for a fixed $x_0$ in the interior of $K$, from where it follows
$$
\sum_{j=1}^{\infty}\left|u_j(f) 2 \pi \lambda_j^{\alpha}\right| < +\infty$$ for every $\alpha \in {\mathbb N}_0^p.$
Consequently $T\left(f\right)\left(x\right) := \sum_{j=1}^{\infty}u_j(f)  e^{2 \pi i x \cdot \lambda_j}$ defines a $C^{\infty}$ function in $\R^p$ and we obtain that $T:C^{\infty}(K) \to C^{\infty}({\mathbb R}^p)$ is a linear extension operator. The continuity of $T$ follows from the Banach-Steinhaus theorem, as   $T\left(f\right)$ is the pointwise limit of $T_n\left(f\right) := \sum_{j = 1}^n u_j\left(f\right)f_j$, $f_j(x):=e^{2\pi i x \cdot \lambda_j}$.
\end{proof}

\par\medskip\noindent
Assume that condition (1) in the previous theorem holds. Then, for a fixed $j_0\in {\mathbb Z}^p$ we can choose $\phi$ such that the $j_0$-th Fourier coefficient of $\phi T(e^{2\pi i \lambda^{j_0}\cdot})$ is not equal to $1.$ According to the comment after Corollary \ref{cor:perturbadualfrechet}, we may remove one of the exponentials in the atomic decomposition above and still obtain an atomic decomposition.

\par\medskip\noindent
Choosing $\psi \neq \phi $ in the proof above,  we find a different sequence $\left(v_j\right) \in C^{\infty}\left(K\right)'$  such that  $\left( \left\{v_j\right\}, \left\{e^{2 \pi i x \cdot \lambda_j }\right\} \right)$ is an unconditional atomic decomposition for $C^{\infty}(K).$ In fact, according to \cite{008}, no system of exponentials can be a basis in $C^{\infty}\left([0,1]\right).$

\vspace{.3cm}

\noindent
\textbf{Example 3.} We give an atomic decomposition of the Schwartz space ${\mathcal S}({\mathbb R}^p)$ of rapidly decreasing functions. It is inspired by the work of Pilipovic, Stoeva and Teofanov \cite{034}, although their Theorem 4.2 cannot be directly applied to conclude that one gets an atomic decomposition.
Let $a,b > 0,$ and $\Lambda = a{\mathbb Z}^p\times b{\mathbb Z}^p$ be given. For $z = (x, \xi)\in {\mathbb R}^{2p}$ and $f\in L^2({\mathbb R}^p)$ we put $\pi(z)f(t) = e^{2\pi i \xi t}f(t-x).$ Let us assume that $g\in {\mathcal S}({\mathbb R}^p)$ and $\left\{\pi(\lambda)g:\ \lambda\in \Lambda\right\}$ is a Gabor frame in $L^2({\mathbb R}^p).$ As proved by Janssen (see \cite[Corollary 11.2.6]{046}) the dual window is also a function $h\in {\mathcal S}({\mathbb R}^p)$ and every $f\in L^2({\mathbb R}^p)$ can be written as
\begin{equation}
 f = \sum_{\lambda\in \Lambda}\left<f, \pi(\lambda)g\right> \pi(\lambda)h.
\end{equation}
\par\medskip\noindent
For every $\lambda \in \Lambda$ we consider $u_{\lambda}\in {\mathcal S}^\prime ({\mathbb R}^p)$ defined by $u_{\lambda}(f) = \left<f, \pi(\lambda)g\right>.$
\begin{proposition}
 $\left((u_{\lambda})_{\lambda\in \Lambda}, (\pi(\lambda)h)_{\lambda\in \Lambda}\right)$ is an unconditional atomic decomposition for ${\mathcal S}({\mathbb R}^p).$
\end{proposition}
\begin{proof}
 According to \cite[Corollary 11.2.6]{046}, the topology of ${\mathcal S}({\mathbb R}^p)$ can be described by the sequence of seminorms
$$
q_n(f):=\sup_{z\in {\mathbb R}^{2p}}\left|\left<f, \pi(z)g\right>\right|v_n(z),\ n\in {\mathbb N},
$$ where $v_n(z) = (1 + |z|)^n.$ So, we only need to check that, for every $n\in {\mathbb N},$
\begin{equation}\label{eq:gabor}
 \sum_{\lambda\in \Lambda}\left|\left<f, \pi(\lambda)g\right>\right| q_n\left(\pi(\lambda)h\right) < \infty.
\end{equation} To this end, we fix $N > n$ large enough. Since $$\left|\left<\pi(\lambda)h, \pi(z)g\right>\right| \leq \left|\left<h, \pi(z-\lambda)g\right>\right| \leq q_N(h)v_N(z-\lambda)^{-1}$$ and $v_n$ is submultiplicative we obtain that (\ref{eq:gabor}) is dominated by
$$
q_N(h)q_N(f)\sum_{\lambda\in \Lambda}\left(v_N(\lambda)\right)^{-1}v_n(\lambda) < \infty
$$ and the proof is finished.
\end{proof}

This example is closely related to the fact that $\left\{\pi(\lambda)g:\ \lambda\in \Lambda\right\}$ is a Gabor frame for each modulation space defined in terms of a polynomially moderate weight; see for instance \cite[Corollary 12.2.6]{046}.

\vspace{.2cm}

\textbf{Acknowledgement.} This research was partially supported by MEC and FEDER Project MTM2010-15200.


\noindent \textbf{Author's address:}%
\vspace{\baselineskip}%

Jos\'e Bonet (corresponding author): Instituto Universitario de Matem\'{a}tica Pura y Aplicada IUMPA, Universitat Polit\`{e}cnica de Val\`{e}ncia,  E-46071 Valencia, Spain. 

email: jbonet@mat.upv.es; phone number: +34963879497; fax number: +34963879494. \\

Carmen Fern\'andez and Antonio Galbis: Departamento de An\'alisis Matem\'atico, Universitat de Val\`encia, E-46100 Burjasot (Valencia), Spain. 

emails: fernand@uv.es, antonio.galbis@uv.es \\

Juan M. Ribera: Instituto Universitario de Matem\'{a}tica Pura y Aplicada IUMPA,
Universitat Polit\`{e}cnica de Val\`{e}ncia,  E-46071 Valencia, Spain.

email: juaripuc@mat.upv.es

\end{document}